\documentclass[a4paper,12pt,oneside,headsepline,draft]{scrartcl}

\usepackage{etoolbox}
\usepackage{ifdraft}

\usepackage[utf8]{luainputenc}
\usepackage[T1]{fontenc}

\usepackage{lmodern}
\usepackage{fourier}
\usepackage{amssymb, amsfonts}

\usepackage{mathrsfs} % script characters (\mathscr)
\usepackage{dsfont}
\usepackage[usenames,dvipsnames]{xcolor}%, colortbl}
\usepackage{lettrine}
\usepackage{url}

\usepackage{stmaryrd}
\usepackage[english]{babel}

\usepackage[draft=false]{scrlayer-scrpage}

\usepackage{wrapfig}
\usepackage{footmisc}
\usepackage{needspace}

\usepackage{amsmath}
\usepackage{nccmath}
\usepackage[thmmarks, amsmath, amsthm]{ntheorem}

\usepackage[pdftex,final,
            pdfborder={0 0 0}, colorlinks=true,
            linkcolor=BrickRed, citecolor=ForestGreen, urlcolor=RoyalBlue]{hyperref}

\ifdraft{%
\usepackage[textsize=scriptsize,colorinlistoftodos]{todonotes}
}{}

\usepackage{booktabs}
\usepackage[inline]{enumitem}
\usepackage{float}
\usepackage{graphicx}
\usepackage{multicol}
\usepackage{scrtime}
\ifdraft{\date{\today}}{\date{}}

\KOMAoptions{DIV=10}

\clearscrheadings
\automark[section]{subsection}

\renewcommand{\subsectionmark}[1]{}

\lohead{{\small \headertitle}}
\rohead{{\small \headerauthors}}

\RedeclareSectionCommand[%
  font=\Large\sffamily\bfseries,%
  beforeskip=1\baselineskip,%
  afterskip=0.5\baselineskip,%
  indent=0em%
  ]{section}

\RedeclareSectionCommands[%
  font=\normalfont\bfseries,%
  afterskip=-1em%
  ]{subsection,subsubsection}

\RedeclareSectionCommands[%
  font=\normalfont\itshape,%
  afterskip=-1em,%
  indent=0pt,
  ]{paragraph}

\newenvironment{enumeratearabic}{
\begin{enumerate}[label=(\arabic*), leftmargin=0pt,labelindent=2em,itemindent=!]
}{
\end{enumerate}
}

\newenvironment{enumerateroman}{
\begin{enumerate}[label=(\roman*), leftmargin=0pt,labelindent=2em,itemindent=!]
}{
\end{enumerate}
}

\newenvironment{enumeratearabic*}{
\begin{enumerate*}[label=(\arabic*)] %%, leftmargin=0pt,labelindent=2em,itemindent=!]
}{
\end{enumerate*}
}

\newenvironment{enumerateroman*}{
\begin{enumerate*}[label=(\roman*)] %%, leftmargin=0pt,labelindent=2em,itemindent=!]
}{
\end{enumerate*}
}

\numberwithin{equation}{section}

\theoremnumbering{arabic}
\newtheorem{theoremcounter}{theoremcounter}[section]
\theoremnumbering{Alph}
\newtheorem{maintheoremcounter}{maintheoremcounter}

\theoremstyle{plain}

\newtheorem{lemma}[theoremcounter]{Lemma}

\newtheorem{proposition}[theoremcounter]{Proposition}
\newtheorem{theorem}[theoremcounter]{Theorem}

\theoremstyle{plain}

\newtheorem{maintheorem}[maintheoremcounter]{Theorem}

\theoremstyle{definition}

\theoremstyle{remark}

\newtheorem{remark}[theoremcounter]{Remark}

\theoremstyle{nonumberremark}

\newtheorem{mainremark}{Remark}

\newtheorem{remarkcomputation}{Computation}

\newenvironment{mainremarkenumerate}
{%
\mainremark
\enumeratearabic
}{%
\endenumeratearabic
\endmainremark
}%

\let\cal\undefined
\ifdraft{
 
}{
 
}

\newcommand{\tx}{\ensuremath{\text}}

\newcommand{\tbf}{\bfseries}

\newcommand{\nbd}{\nobreakdash-\hspace{0pt}}

\newcommand{\bboard}{\ensuremath{\mathbb}}
\newcommand{\cal}{\ensuremath{\mathcal}}

\newcommand{\bbM}{\ensuremath{\bboard M}}

\newcommand{\cO}{\ensuremath{\cal{O}}}

\newcommand{\rmd}{\ensuremath{\mathrm{d}}}

\newcommand{\rmf}{\ensuremath{\mathrm{f}}}

\newcommand{\rms}{\ensuremath{\mathrm{s}}}

\newcommand{\rmM}{\ensuremath{\mathrm{M}}}

\newcommand{\rmU}{\ensuremath{\mathrm{U}}}

\newcommand{\td}{\tilde}

\newcommand{\ov}{\overline}

\newcommand{\ra}{\ensuremath{\rightarrow}}

\newcommand{\amid}{\ensuremath{\mathop{\mid}}}

\newcommand{\ZZ}{\ensuremath{\mathbb{Z}}}
\newcommand{\QQ}{\ensuremath{\mathbb{Q}}}

\newcommand{\CC}{\ensuremath{\mathbb{C}}}

\renewcommand{\Im}{\ensuremath{\mathrm{Im}}}

\newcommand{\isdiv}{\amid}
\newcommand{\nisdiv}{\ensuremath{\mathop{\nmid}}}

\renewcommand{\pmod}[1]{\ensuremath{\;(\mathrm{mod}\, #1)}}

\newenvironment{psmatrix}{\left(\begin{smallmatrix}}{\end{smallmatrix}\right)}

\newcommand{\GL}[1]{\ensuremath{\mathrm{GL}_{#1}}}

\newcommand{\SL}[1]{\ensuremath{\mathrm{SL}_{#1}}}

\renewcommand{\det}{\ensuremath{\mathrm{det}}}

\newcommand{\HS}{\mathbb{H}}

\newcommand{\ga}{\ensuremath{\gamma}}
\newcommand{\Ga}{\ensuremath{\Gamma}}

\newcommand{\hol}{\ensuremath{\mathrm{hol}}}

\usepackage[backend=bibtex]{biblatex}
\usepackage[english]{babel}  

\newcommand{\ord}{\operatorname{ord}}
\renewcommand{\max}{\operatorname{max}}

\newcommand{\headertitle}{{\normalfont%
  Congruences of Hurwitz class numbers%
}}
\newcommand{\headerauthors}{%
  O.~Beckwith, M.~Raum, O.~K.~Richter%
}
\title{%
  Congruences of Hurwitz class numbers on square classes
}
\author{%
Olivia Beckwith%
\and
Martin Raum%
\thanks{The author was partially supported by Vetenskapsr\aa det Grants~2015-04139 and~2019-03551.}%
\and
Olav K. Richter
\thanks{The author was partially supported by Simons Foundation Grants~\#412655 and~\#835652.}
}
\date{\today\ at\ \thistime}

\begin{document}

\thispagestyle{scrplain}
\begingroup
\deffootnote[1em]{1.5em}{1em}{\thefootnotemark}
\maketitle
\endgroup

%%%%%%%%%%%%%%%%%%%%%%%%%%%%%%%%%%%%%%%%%%%%%%%%%%
%%% ABSTRACT

\begin{abstract}
\small
\noindent
{\tbf Abstract:}
We extend a holomorphic projection argument of our earlier work to prove a novel divisibility result for non-holomorphic congruences of Hurwitz class numbers. This result allows us to establish Ra\-ma\-nu\-jan-type congruences for Hurwitz class numbers on square classes, where the holomorphic case parallels previous work by Radu on partition congruences. We offer two applications. The first application demonstrates common divisibility features of Ra\-ma\-nu\-jan-type congruences for Hurwitz class numbers.  The second application provides a dichotomy between congruences for class numbers of imaginary quadratic fields and Ra\-ma\-nu\-jan-type congruences for Hurwitz class numbers.
\\[.3\baselineskip]
\noindent
\textsf{\textbf{%
  Hurwitz class numbers%
}}%
\noindent
{\ {\tiny$\blacksquare$}\ }%
\textsf{\textbf{%
  Ramanujan-type congruences%
}}%
\noindent
{\ {\tiny$\blacksquare$}\ }%
\textsf{\textbf{%
  holomorphic projection%
}}
\\[.2\baselineskip]
\noindent
\textsf{\textbf{%
  MSC Primary:
  11E41%
}}%
{\ {\tiny$\blacksquare$}\ }%
\textsf{\textbf{%
  MSC Secondary:
  11F33, 11F37%
}}
\end{abstract}

%%%%%%%%%%%%%%%%%%%%%%%%%%%%%%%%%%%%%%%%%%%%%%%%%%
%%% TABLE OF CONTENTS
%
% \vspace{-1.5em}
% \renewcommand{\contentsname}{}
% \setcounter{tocdepth}{2}
% \tableofcontents
% \vspace{1.5em}

%%%%%%%%%%%%%%%%%%%%%%%%%%%%%%%%%%%%%%%%%%%%%%%%%%
%%% INTRODUCTION

\Needspace*{4em}
\addcontentsline{toc}{section}{Introduction}
\markright{Introduction}

\lettrine[lines=2,nindent=.2em]{\tbf H}{urwitz} class numbers~$H(D)$ play a significant role in classical number theory. If~$-D < -4$ is a negative fundamental discriminant, then~$h(-D) = H(D)$ is the class number of the imaginary quadratic field~$\QQ(\sqrt{-D})$. Despite intensive studies, divisibility properties of these class numbers have remained mysterious.  

In this work, we investigate Ramanujan-type congruences for Hurwitz class numbers (see Theorem~\ref{mainthm:nonholomorphic-ell-divides-b}, \ref{mainthm:square-classes}, and~\ref{mainthm:a-b-ppowers}).  Furthermore, in Theorem~\ref{mainthm:dichotomy} we connect Ramanujan-type congruences for Hurwitz class numbers~$H(D)$ to congruences for class numbers~$h(-D)$ in certain families of fundamental discriminants~$-D$.

In~\cite{beckwith-raum-richter-2020}, we explored Ramanujan-type congruences for Hurwitz class numbers~$H(D)$ such as the following examples:
\begin{gather*}
%\label{eq:nonholomorphic-Ramanujan-examples}
\begin{aligned}
  H(5^3 n + 5^2) &\equiv 0 \;\pmod{5}
\tx{,}\\
  H(7^3 n + 3 \cdot 7^2) &\equiv 0 \;\pmod{7}
\tx{,}\\
  H(11^3 n + 7 \cdot 11^2) &\equiv 0 \;\pmod{11}
\tx{.}
\end{aligned}
\end{gather*}
These congruences are of the form~$H(an + b) \equiv 0\, \pmod{\ell}$, where $\ell>3$ is a prime and $a > 0$ and~$b$ are integers such that~$-b$ is a square modulo~$a$. We refer to such congruences as non-holomorphic Ramanujan-type congruences, because  the generating series for $H(an+b)$ is a mock modular form, i.e., it has a non-holomorphic modular completion. In particular, one cannot access such congruences via standard techniques from the theory of holomorphic modular forms.

In our earlier work~\cite{beckwith-raum-richter-2020} we employed a holomorphic projection argument to prove that for such non-holomorphic congruences the divisibility $\ell \isdiv a$ holds. The above examples also suggest the divisibility $\ell \isdiv b$, and in this current paper we use another holomorphic projection argument to prove:

\begin{maintheorem}
\label{mainthm:nonholomorphic-ell-divides-b}
Let $\ell>3$ be a prime, $a \in \ZZ_{\ge 1}$, and $b \in \ZZ$. If\/ $-b$ is a square modulo~$a$ and
\begin{gather*}
  H(a n + b) \equiv 0 \;\pmod{\ell}
\end{gather*}
for all integers~$n$, then~$\ell \isdiv b$.
\end{maintheorem}

% Ramanujan-type congruences for Hurwitz class numbers on~$a \ZZ + b$ doe not merely occur for single~$b$.
There are also holomorphic Ramanujan-type congruences for Hurwitz class numbers, i.e., congruences~$H(a n + b) \equiv 0 \,\pmod{\ell}$ where $-b$ is not a square modulo~$a$: 
\begin{gather*}
%\label{eq:holomorphic-Ramanujan-examples}
\begin{aligned}
  H(3^3 n + 3^2) &\equiv 0 \;\pmod{5}
\tx{,}\\
  H(5^3 n + 2 \cdot 5^2) &\equiv 0 \;\pmod{7}
\tx{,}\\
  H(2^9 n + 3 \cdot 2^6) &\equiv 0 \;\pmod{11}
\tx{.}
\end{aligned}
\end{gather*}
In these examples~$\ell$ does not divide~$a$ or $b$. While such congruences can be studied with tools from the theory of holomorphic modular forms, the relation between~$a$ and~$b$ has not yet been resolved, either.

Our examples of holomorphic and non-holomorphic congruences indicate that $\ord_p\big(a \slash \gcd(a,b)\big) \le 1$ for odd primes $p$ and $\ord_2\big(a \slash \gcd(a,b)\big) \le 3$. To prove these phenomena, we first establish the following result on congruences of square classes:

\begin{maintheorem}
\label{mainthm:square-classes}
Let $\ell>3$ be a prime, $a \in \ZZ_{\ge 1}$, and $b \in \ZZ$. Suppose~$H(an + b) \equiv 0 \,\pmod{\ell}$ for all integers~$n$. Then $H(an + b u^2) \equiv 0 \,\pmod{\ell}$ for all integers~$u$ with $\gcd(u,a)=1$ and $n \in \ZZ$. 
\end{maintheorem}

In the case of holomorphic Ramanujan-type congruences, the proof of Theorem~\ref{mainthm:square-classes} extends ideas of Radu's study of partition congruences~\cite{radu-2013,radu-2013}.  The proof of Theorem~\ref{mainthm:square-classes} in the non-holomorphic case requires a deeper analysis, where Theorem~\ref{mainthm:nonholomorphic-ell-divides-b} is an essential ingredient.

We give two applications of Theorem~\ref{mainthm:square-classes}. In our first application, we call a Ra\-ma\-nu\-jan-type congruences for~$H(D)$ modulo~$\ell$ on~$a \ZZ + b$ \emph{maximal}, if~$H(D)$ has no Ramanujan-type congruence modulo~$\ell$ on any arithmetic progression~$a' \ZZ + b'$ that is properly contained in~$a \ZZ + b$.

\begin{maintheorem}
\label{mainthm:a-b-ppowers}
Let $\ell>3$ be a prime. Suppose that we have a maximal Ramanujan-type congruence modulo~$\ell$ for the Hurwitz class numbers on~$a \ZZ + b$. Then for odd primes~$p$,
\begin{gather*}
  \ord_p\big( a \slash \gcd(a,b) \big) \le 1
\quad\tx{and}\quad
  \ord_2\big( a \slash \gcd(a,b) \big) \le 3
\tx{.}
\end{gather*}
\end{maintheorem}

As a further application of Theorem~\ref{mainthm:square-classes}, we provide a dichotomy between Ra\-ma\-nu\-jan-type congruences for Hurwitz class numbers and congruences for class numbers of imaginary quadratic fields whose discriminant varies in a square class modulo~$a$. For the next statement, we require the usual Legendre symbol and also the divisor sum~$\sigma_1(b) := \sum_{d \isdiv b} d$. For a prime~$p$ and an integer~$a$, we call the largest~$p$\nbd power that divides~$a$ its~$p$\nbd part.

\begin{maintheorem}
\label{mainthm:dichotomy}
Let $\ell>3$ be a prime. Suppose that we have a Ramanujan-type congruence modulo~$\ell$ for the Hurwitz class numbers on~$a \ZZ + b$. For all odd primes~$p \isdiv a$ assume that~$\ord_p\big( a \slash \gcd(a,b) \big) \ge 1$ and if~$a$ is even, assume that~$\ord_2\big( a \slash \gcd(a,b) \big) \ge 2$. Then either:
\begin{enumerateroman}
\item
\label{it:mainthm:dichotomy:fundamental-discriminants}
We have~$h(-D) \equiv 0 \,\pmod{\ell}$ for all fundamental discriminants~$-D < -4$ for which there is~$f \in \ZZ\setminus \{0\}$ with~$D f^2 \in a \ZZ + b$.

\item
\label{it:mainthm:dichotomy:hecke-eigenvalues}
There is a prime~$p$ dividing~$a$ such that 
\begin{gather*}
  \sigma_1(f_p)
\equiv
  \left(\frac{-D}{p}\right)
  \sigma_1(f_p \slash p)
  \;\pmod{\ell}
\end{gather*}
for every fundamental discriminant~$-D < 0$ and integer $f$ satisfying $D f^2 \equiv b \pmod{a}$, where $f_p$ is the $p$-part of $f$. Both~$(-D \slash p)$ and~$f_p$ are uniquely determined by~$a \ZZ + b$. In this case, we have a Ramanujan-type congruence for Hurwitz class numbers on~$a_p \ZZ + b$, where~$a_p$ is the $p$-part of~$a$. 
\end{enumerateroman}

Additionally, if the congruence in~\ref{it:mainthm:dichotomy:hecke-eigenvalues} holds for any~$Df^2 \in a \ZZ + b$, then it holds for all~$Df^2 \in a \ZZ + b$ for which~$-D f^2$ is a discriminant and we have a Ramanujan-type congruence modulo $\ell$ on $a\ZZ+b$.
\end{maintheorem}

% \begin{mainexample}
% holds for $D=3$ and $f=7$, and the congruence
% \begin{gather*}
%   H(7^3 n + 3 \cdot 7^2) \equiv 0 \;\pmod{7}
% \end{gather*}
% follows from the last part of Theorem \ref{mainthm:dichotomy}.
% \end{mainexample}

\begin{mainremarkenumerate}
\item
\label{mainrm:mainthm:dichotomy:assumptions_ord}
The assumptions on the orders of~$a \slash \gcd(a,b)$ can always be achieved by replacing~$a$ with a suitable multiple of it. They can be removed at the expense of a more technical statement involving the factorizations~$D f^2 \in a \ZZ + b$ that appear in case~\ref{it:mainthm:dichotomy:fundamental-discriminants}.

\item
From the Hurwitz class number formula alone, one could deduce a statement similar to case~\ref{it:mainthm:dichotomy:hecke-eigenvalues} for some prime~$p$, not necessarily dividing~$a$. To show that one must have~$p \isdiv a$, we use Theorem \ref{mainthm:square-classes}. 

\item
One can verify that all congruences given in this introduction fall under case~\ref{it:mainthm:dichotomy:hecke-eigenvalues}. We do not expect that the first case in the theorem ever occurs, i.e., we expect that Ramanujan-type congruences for~$H(D)$ modulo~$\ell$ on~$a \ZZ + b$ occur if and only if there is~$p \isdiv a$ and~$Df^2 \in a \ZZ + b$ satisfying the above condition in~\ref{it:mainthm:dichotomy:hecke-eigenvalues}. This belief is supported by extensive numerical evidence in addition to well-known theorems on the divisibility of class numbers, such as \cite{Wiles}, which implies that if the first case occurs, we must have~$p \equiv \pm 1 \,\pmod{\ell}$ for any odd primes $p$ with~$2 \nisdiv \ord_p(a)$. 

\item For any congruences for Hurwitz class numbers that do not fall under case~\ref{it:mainthm:dichotomy:hecke-eigenvalues}, Theorem~\ref{mainthm:dichotomy} implies the existence of fundamental discriminants~$-D \equiv u^2 b \pmod{a}$ for which~$h(-D) \equiv 0 \,\pmod{\ell}$ for some~$u$ co-prime to~$a$.
\end{mainremarkenumerate}

The proof of Theorem~\ref{mainthm:dichotomy} relies on the Hurwitz class number formula and Theorem~\ref{mainthm:square-classes}.  In accordance with Theorem~\ref{mainthm:square-classes}, the case of holomorphic Ramanujan-type congruences is accessible via methods from the classical theory of modular forms, while the non-holomorphic case is not.

The paper is organized as follows.  In Section~\ref{sec:preliminaries}, we review some tools from the theory of modular forms needed for our work.  In Section~\ref{sec:non-holomorphic-congruences}, we establish Theorem~\ref{mainthm:nonholomorphic-ell-divides-b}.  In Section~\ref{sec:square-classes} we prove Theorem~\ref{mainthm:square-classes}.  Finally,  in Section~\ref{sec:applications} we settle Theorems~\ref{mainthm:a-b-ppowers} and~\ref{mainthm:dichotomy}.

%%%%%%%%%%%%%%%%%%%%%%%%%%%%%%%%%%%%%%%%%%%%%%%%%%
%%% PAPER BODY

\section{Preliminaries}%
\label{sec:preliminaries}

We introduce necessary notation to discuss modular forms (see for example \cite{bringmann-folsom-ono-rolen-2018}) and quasi-modular forms (see for example \cite{zagier-1994,kaneko-zagier-1995}).  For odd~$D$, set
\begin{gather*}
  \epsilon_D
=
  \begin{cases}
  1\tx{,} & \tx{if $D \equiv 1 \,\pmod{4}$;} \\
  i\tx{,} & \tx{if $D \equiv 3 \,\pmod{4}$.}
  \end{cases}
\end{gather*}

Throughout the paper $\tau\in\HS$ (the usual complex upper half plane), $y=\Im(\tau)$, and $e(s\tau) := \exp(2 \pi i\, s\tau)$ for $s\in\QQ$. Let~$\Ga_0(N)$, $\Ga_1(N)$, and $\Ga(N)$ be the standard congruence subgroups of~$\SL{2}(\ZZ)$. Let $\rmM_k(\Ga)$ denote the space of modular forms of integral or half-integral weight~$k$ for~$\Ga \subseteq \SL{2}(\ZZ)$ with respect to the multiplier $\nu_{\theta}^{2k}$ if $k \not\in \mathbb{Z}$ (where $\nu_{\theta}$ is the theta multiplier), and $\bbM_k(\Ga)$ the corresponding space of harmonic Maass forms (satisfying the moderate growth condition at all cusps).

For $\ga = \begin{psmatrix} a & b \\ c & d \end{psmatrix} \in \GL{2}(\QQ)$ with~$\det(\ga) > 0$, the weight-$k$ slash operator is defined by
\begin{gather*}
  \big( f \big|_k \ga \big) (\tau)
\;=\;
  (\det \ga)^{-\frac{k}{2}}\,
  (c \tau + d)^{-k}\,
  f \big(\mfrac{a \tau + b}{c \tau + d}\big)
\tx{.}
\end{gather*}

Recall that if~$f(\tau)=\sum_{m\geq 0}c(f,m)e(m\tau) \in \rmM_{2-k}(\Ga(N))$ is a holomorphic modular form of level~$N \in \ZZ_{\ge 1}$ with~$k \ne 1$, then its non-holomorphic Eichler integral is given by
\begin{gather}
\label{eq:def:non-holomorphic-eichler-integral}
\begin{aligned}
  f^\ast(\tau)
\;&:=\;
  -(2 i)^{k-1}
  \int_{-\ov{\tau}}^{i\infty} \frac{\ov{f(-\ov{w})}}{(w + \tau)^k}\,d\!w
\\
&\hphantom{:}=
  \frac{\ov{c(f,0)}}{1-k}\, y^{1-k}
  \,-\,
  (4 \pi)^{k-1}
  \sum_{m < 0}
  \ov{c(f, |m|)\, }|m|^{k-1}
  \Gamma(1-k,4 \pi |m| y) e(m\tau)
\tx{,}
\end{aligned}
\end{gather}
where $\Ga$ represents the upper incomplete Gamma-function.

\subsection{Generating series of Hurwitz class numbers}

Zagier~\cite{zagier-1975} investigated the generating series~$\sum_D H(D) e(D \tau)$ of Hurwitz class numbers, and proved that it has a modular completion:
\begin{gather}
\label{eq:zagier-eisenstein-series}
  E_{\frac{3}{2}}(\tau)
\;:=\;
  \sum_{D = 0}^\infty H(D) e(D \tau)
\,+\,
  \mfrac{1}{16 \pi} \theta^\ast(\tau)
\;\in\;
  \bbM_{\frac{3}{2}}(\Gamma_0(4))
\tx{,}
\end{gather}
where
\begin{gather}
\label{eq:theta-modularity}
\begin{aligned}
  \theta
&\;:=\;
  \theta_{1,0}
\,\in\,
  \rmM_{\frac{1}{2}}(\Ga_0(4))
\tx{\, with}
\\
  \theta_{a,b}(\tau)
&\;:=\;
  \hspace{-1em}
  \sum_{\substack{n \in \ZZ\\n \equiv b \,\pmod{a}}}
  e\big( \tfrac{n^2 \tau}{a} \big)
\,\in\,
  \rmM_{\frac{1}{2}}(\Ga(4 a))
\tx{,}\quad
  a \in \ZZ_{\ge 1}, b \in \ZZ
\tx{.}
\end{aligned}
\end{gather}

For $a \in \ZZ_{\ge 1}$ and $b \in \ZZ$, we recall the operators~$\rmU_{a,b}$ from our earlier work~\cite{beckwith-raum-richter-2020}, which act on Fourier series expansions of non-holomorphic modular forms by:
\begin{gather}
  \rmU_{a,b}\,\sum_{n \in \ZZ} c(f;\,n;\,y) e(n \tau)
\;:=\;
  \sum_{\substack{n \in \ZZ\\n \equiv b \,\pmod{a}}}
  c\big( f;\,n;\,\tfrac{y}{a} \big)
  e\big( \tfrac{n \tau}{a} \big)
\tx{.}
\end{gather}

In particular, the holomorphic part of $\rmU_{a,b}\, E_{\frac{3}{2}}(\tau)$ is the generating series of Hurwitz class numbers~$H(a n + b)$ for~$n \in \ZZ$, and one finds that (see also~\cite{cohen-1975,jochnowitz-2004-preprint} for the holomorphic case)
\begin{gather}
\label{eq:eisenstein-ab-modularity}
  \rmU_{a,b}\, E_{\frac{3}{2}}
\,\in\,
  \bbM_{\frac{3}{2}}(\Ga(4 a))
\tx{.}
\end{gather}

The action of the~$\rmU$\nbd operators on theta series can be described by
\begin{gather}
\label{eq:Uab-theta}
  \rmU_{a,b}\,\theta
\;=\;
  \sum_{\beta^2 \equiv b \,\pmod{a}}
  \theta_{a,\beta}
\quad\tx{and}\quad
  \rmU_{a,b}\,\theta^\ast
\;=\;
  \sum_{\beta^2 \equiv -b \,\pmod{a}}
  \sqrt{a}\, \theta^\ast_{a, \beta}
\tx{.}
\end{gather}
Note that if~$-b$ is not a square modulo~$a$, then $\rmU_{a,b}\,E_{\frac{3}{2}}$ is a holomorphic modular form.

\subsection{Holomorphic projection}

Holomorphic projection plays an important role in our proofs of Theorems~\ref{mainthm:nonholomorphic-ell-divides-b} and~\ref{mainthm:square-classes}.  We briefly review the holomorphic projection operator from~\cite{imamoglu-raum-richter-2014} in the scalar-valued case (see also~\cite{sturm-1980,gross-zagier-1986}).  Let~$k \in \ZZ$, $k \ge 2$, $N \in \ZZ_{\ge 1}$, and $f :\,\HS \ra \CC$ an~$N$-periodic continuous function with Fourier series expansion
\begin{gather*}
  f(\tau)
\;=\;
  \sum_{n \in \frac{1}{N}\ZZ}
  c(f;\,n;\,y) e(\tau n)
\end{gather*}
satisfying the conditions:
%\begin{enumerateroman}
\begin{enumerateroman*}
\item
For some $a > 0$ and all~$\ga \in \SL{2}(\ZZ)$, there are coefficients $\td{c}(f|_k\,\ga;\,0) \in \CC$, such that $(f |_k\,\ga)(\tau) = \td{c}(f|_k\,\ga;\,0) + \cO(y^{-a})$ as $y \ra \infty$;

\item
For all $n \in \frac{1}{N}\ZZ_{> 0}$, we have $c(f;\, n;\, y) = \cO(y^{2-k})$ as $y \ra 0$.
\end{enumerateroman*}
%\end{enumerateroman}
Then
\begin{gather}
\label{eq:def:holomorphic-projection}
\begin{aligned}
  \pi^\hol_k(f)
\;&:=\;
  \td{c}(f;\,0)
  \,+\,
  \sum_{n \in \frac{1}{N}\ZZ_{>0}}
  c\big(\pi^\hol_k(f);\, n\big)\, e(n\tau)
\quad\tx{with}
\\
  c\big(\pi^\hol_k(f);\, n\big)
\;&:=\;
  \frac{(4 \pi n)^{k-1}}{\Gamma(k-1)}\,
  \lim_{s \ra 0}\,
  \int_0^\infty
  c(f;\, n;\, y) \exp(-4 \pi n y) y^{s+k-2}\,\rmd\!y
\tx{.}
\end{aligned}
\end{gather}

Recall the following key properties of the holomorphic projection operator in~\eqref{eq:def:holomorphic-projection}:  Proposition~4 of~\cite{imamoglu-raum-richter-2014} states that if $f$ is holomorphic, then $\pi^\hol_k(f) = f$. Theorem~5 of~\cite{imamoglu-raum-richter-2014} asserts that if $f$ transforms like a modular form of weight~$2$ for the group~$\Ga_1(N)$, then~$\pi^\hol_2(f)$ is a quasi-modular form of weight~$2$ for $\Ga_1(N)$. 

\subsection{A theorem of Serre}

We conclude this Section with a result of Serre, which is required for our proof of Theorem~\ref{mainthm:nonholomorphic-ell-divides-b}.

\begin{theorem}[Serre \cite{serre-1973}, \cite{serre-1974}]
\label{thm:serre}
Fix positive integers~$k$ and~$N$, and a prime number~$\ell$. Then there exist infinitely many primes~$p \equiv 1 \pmod{\ell N}$ such that for all $f \in M_k(\Gamma_1(N))$ with $\ell$-integral Fourier coefficients, we have
\begin{gather*}
  c( f; np^r ) \equiv (r+1)\, c(f;n) \,\pmod{\ell}
\end{gather*}
for all $n \in \ZZ$ coprime to $\ell$ and all non-negative integers~$r$.
\end{theorem}

\section{Conditions on non-holomorphic congruences}%
\label{sec:non-holomorphic-congruences}

We have already proved in~\cite{beckwith-raum-richter-2020} that non-holomorphic congruences modulo~$\ell$ for Hurwitz class numbers on an arithmetic progression~$a \ZZ + b$ include the divisibility~$\ell \isdiv a$. For the purpose of this work, we need to extend this result.

We will prove Theorem~\ref{mainthm:nonholomorphic-ell-divides-b} by contraposition. Proposition~\ref{prop:coefficient-formula} provides us with explicit congruences for specific Fourier series condition, which we then use to derive a contradiction. The next lemma allows us to pass from a given arithmetic progression~$\td{a} \ZZ + \td{b}$ to a more convenient one.
\begin{lemma}
\label{la:subprogression}
Let~$\td{a} \in \ZZ_{\ge 1}$ and~$\td{b},\beta \in \ZZ$ be such that~$-\td{b} \equiv =\beta^2 \pmod{\td{a}}$. Then there exist~$a \in \ZZ_{\ge 1}$ and $b \in \ZZ$ such that 
\begin{enumerateroman}
\item
\label{it:la:subprogression:first-condition}
\label{it:la:subprogression:a-divides-and-b-congruent}
We have~$\td{a} \isdiv a$ and~$b \equiv \td{b} \,\pmod{\td{a}}$.

\item
\label{it:la:subprogression:-b-is-square}
We have ~$-b \equiv \beta^2 \,\pmod{a}$.

\item
\label{it:la:subprogression:proper-local-divisors}
For every prime~$p \isdiv a$, writing~$a_p$ for the~$p$\nbd part of~$a$, we have that~$\gcd(a_p, 2 \beta)$ is a proper divisor of~$a_p$.

\item
\label{it:la:subprogression:last-condition}
\label{it:la:subprogression:large-prime-divides-a}
There is a prime~$p \isdiv a$ such that~$a < p^2$ and~$0 \le 2 \beta < p$ .
\end{enumerateroman}
\end{lemma}
\begin{proof}
First we choose an appropriate multiple of $\td{a}$:
\begin{gather*}
  a'
:=
  \prod_{p\isdiv \td{a}} p^{\max \{\ord_p(\td{a)},\, \ord_p (2 \beta) + 1\}}
=
  \mathrm{lcm}(\td{a},\, 2 \beta a_{\rms\rmf} )
\tx{,}
\end{gather*}
where~$a_{\rms\rmf}$ is the maximal square-free divisor of~$a$. We let $p > \max \{ a', 2\beta \}$, and set $a:= a' \cdot p$. Then if $b$ is any integer congruent to $- \beta^2$ modulo $a$, the four requirements in the lemma are met. 
\end{proof}

\begin{proposition}
\label{prop:coefficient-formula}
Assume that~$H(an+b) \equiv 0 \pmod{\ell}$ for all $n$. Further, assume that $\beta^2 \equiv -b \,\pmod{a}$ and~$b \not\equiv 0 \,\pmod{\ell}$. Assume that~$a$, $b$, and~$\beta$ satisfy Conditions~\ref{it:la:subprogression:first-condition}--\ref{it:la:subprogression:last-condition} in Lemma~\ref{la:subprogression}. Set~$a' = \gcd(a, 2 \beta)$.

Then~$\pi^\hol_2\big((  \rmU_{a,b}\,E_{\frac{3}{2}}) \cdot  (\theta_{a,\beta} + \theta_{a,-\beta}) \big)$ is a quasi-modular form for~$\Ga_1(4 a)$ and
\begin{gather*}
  \pi^\hol_2\big((
  \rmU_{a,b}\,E_{\frac{3}{2}})
  \cdot
  (\theta_{a,\beta} + \theta_{a,-\beta})
  \big)
\;=\;
  \sum_{n = 0}^{\infty} c(n) e(n \tau)
\tx{,}
\end{gather*}
where 
\begin{gather*}
  c(a' p)
\;\equiv\;
  0
  \;\pmod{\ell}
\tx{,}\quad
  c(a' p p')
\;\equiv\;
  - a'
  \;\pmod{\ell}
\quad
  \tx{ if\/ $2 \beta \not\equiv a' \,\pmod{a}$}
\end{gather*}
for any primes~$p, p'$ with
\begin{gather}
\label{eq:eisenstein-theta-holomorphic-projection-coefficient-prime-condition:p}
  a' p \equiv 2 \beta \,\pmod{a} 
\tx{,}\;
  p > a \slash a'
\tx{,}
\quad
  p' \equiv 1 \,\pmod{a}
\tx{,}\;
  p' > p a \slash a'
\tx{;}
\end{gather}

And 
\begin{gather*}
  c(a')
\;\equiv\;
  c(a' p')
\;\equiv\;
  - a'
  \;\pmod{\ell}
\quad
  \tx{ if\/ $2 \beta \equiv a' \,\pmod{a}$}
\end{gather*}
for any prime~$p'$ with
\begin{gather}
\label{eq:eisenstein-theta-holomorphic-projection-coefficient-prime-condition:no-p}
  p' \equiv 1 \,\pmod{a}
\tx{,}\;
  p' > a \slash a'
\tx{.}
\end{gather}
\end{proposition}
\begin{remark}
Our Proposition~\ref{prop:coefficient-formula} is the analogue of Proposition~2.5 of~\cite{beckwith-raum-richter-2020}. The assumption~$\ell \nisdiv a$ was accidentally omitted from Proposition~2.5 of~\cite{beckwith-raum-richter-2020}. The analogue to it in our current Proposition~\ref{prop:coefficient-formula} is the condition~$\beta \not\equiv 0 \,\pmod{\ell}$.

There was another issue in the proof of Proposition~2.5 of~\cite{beckwith-raum-richter-2020}. Namely, when arguing that we may assume that~$d_1$ and~$d_2$ are positive (as we will do in the present proof), this is only legitimate when considering the sum over over~$\pm \beta$ and all~$\td\beta$, but not on the level of individual terms.

Finally, on a related note, we remark that the condition that~$b \equiv \td{b} \,\pmod{\td{a}}$ was incorrectly omitted from Lemma~2.3 of~\cite{beckwith-raum-richter-2020}.
\end{remark}

\begin{proof}[Proof of Proposition~\ref{prop:coefficient-formula}]
While we have remarked that the assumption~$\ell \nisdiv a$ was incorrectly omitted from Proposition~2.5 of~\cite{beckwith-raum-richter-2020}, the first part of its proof never makes use of it. The following computation is verbatim the one in~\cite{beckwith-raum-richter-2020}. As in that paper, we note that
\begin{gather*}
  \pi^\hol_2\big(
  (\rmU_{a,b}\,E_{\frac{3}{2}})
  \cdot
  (\theta_{a,\beta} + \theta_{a,-\beta})
  \big)
\;\equiv\;
  \mfrac{1}{16 \pi}\,
  \pi^\hol_2\big(
  (\rmU_{a,b}\,\theta^\ast)
  \cdot
  (\theta_{a,\beta} + \theta_{a,-\beta})
  \big)
  \;\pmod{\ell}
\tx{.}
\end{gather*}
This leads us to computing the sum over~$\pm \beta$ and~$\td{\beta}^2 \equiv -b \,\pmod{a}$ of
\begin{gather}
\label{eq:coefficient-cn-d12-factorizaton}
  c\Big( \pi^\hol_2\big(
  \sqrt{a}\, \theta^\ast_{a,\td\beta}(\tau) \cdot \theta_{a,\beta}(\tau)
  \big);\,
  n
  \Big)
\;=\;
  -4 \pi
  a n\hspace{-1em}
  \sum_{\substack{m \equiv \beta \,\pmod{a}\\\td{m} \equiv \td{\beta} \,\pmod{a}\\\td{m} \ne 0\\a n = m^2 - \td{m}^2}}
  \frac{1}{|m| + |\td{m}|}
\tx{.}
\end{gather}
The term with~$\delta_{\td\beta \equiv 0}$ in Equation~[16] of~\cite{beckwith-raum-richter-2020} does not appear here for the following reason: First, the assumption~$H(a n + b) \equiv 0 \,\pmod{\ell}$ for all integers~$n$ implies~$\ell \isdiv a$ by the main theorem of~\cite{beckwith-raum-richter-2020}. Second, we have~$\td\beta \not\equiv 0 \,\pmod{\ell}$, since~$- \beta^2 \equiv b \not \equiv 0 \pmod{\ell}$. 

We can still proceed as in~\cite{beckwith-raum-richter-2020}, factoring~$an = d_1 d_2$ to arrive at the analogue of Equation~[17]. We treat only the positive case, $d_1, d_2 > 0$; the negative case yields the same sum, after applying the summation over~$\pm \beta$ and all~$\td\beta$. We account for this suppressing the sum over~$\pm \beta$ and multiplying with~$2$. As in~\cite{beckwith-raum-richter-2020}, we assume that~$a n$ is not a square, and obtain after taking the factors~$1 \slash 16 \pi$ and~$-4 \pi$ from our previous expressions in to account that
\begin{gather}
\label{eq:eisenstein-theta-holomorphic-projection-coefficient}
  c(n)
\;=\;
  -\mfrac{1}{2}
  \sum_{\td\beta^2 = -b \,\pmod{a}}\hspace{-0.8em}
  \sum_{\substack{
    a n = d_1 d_2\\%
    d_1, d_2 > 0\\%
    d_1 \equiv \beta + \td\beta \,\pmod{a}\\%
    d_2 \equiv \beta - \td\beta \,\pmod{a}}}
  \big( d_1 \delta_{d_1 < d_2} + d_2 \delta_{d_2 < d_1} \big)
\tx{.}
\end{gather}

Only now we diverge from~\cite{beckwith-raum-richter-2020}, where we separated the archimedean and nonarchimedean conditions in this sum. This is no longer possible in the present setting, but we can still separate all nonarchemedian conditions away from~$\ell$ from the archimedean ones. 

In the following discussion, $q$ will always denote a prime. To ease the discussion, we introduce additional notation: Recall from Lemma~\ref{la:subprogression} that, given~$q \isdiv a$, we denote by~$a_q$ the maximal $q$-power that divides~$a$.  We write~$Q_a$ for the set of all prime divisors of~$a$. Given a subset~$Q \subseteq Q_a$, we let~$a_Q$ be the product of all~$a_q$ for~$q \in Q$, and set~$a_Q^\# := a \slash a_Q$. Likewise, we define $a'_q$ to be the maximal $q$ power dividing $a'$, and we define $a'_Q := \prod_{q \in Q} a'_q$ and~$a_Q^{\prime\,\#}:= a' / a'_Q$.

If $\td{\beta}^2 \equiv \beta^2 \pmod{a}$, then for each $q \in Q_a$, $\td{\beta} \equiv \pm \beta \pmod{a_q}$. Moreover, by the Chinese Remainder Theorem, we can associate to each subset $Q \subseteq Q_a$ a residue class $\td{\beta}_Q \pmod{a}$ such that $\td{\beta}_Q \equiv \beta \pmod{a_q}$ for each $q \in Q$ and $\td{\beta}_Q \equiv - \beta \pmod{a_q}$ for each $q \in Q_a \backslash Q$. Using this correspondence between subsets of $Q_a$ and residue classes $\td{\beta} \pmod{a}$ such that $\td{\beta}^2 \equiv - b \pmod{a}$, we rewrite our expression for $c(n)$:
\begin{gather}
\label{eq:coefficient-Q-factorizaton}
  c(n)
\;=\;
  - \mfrac{1}{2}
  \sum_{Q \subseteq Q_a}\hspace{-1em}
  \sum_{\substack{
    a n = d_1 d_2\\%
    d_1, d_2 > 0\\%
    d_1 \equiv \beta + \td\beta_Q \,\pmod{a}\\%
    d_2 \equiv \beta - \td\beta_Q \,\pmod{a}}}
  \big( d_1 \delta_{d_1 < d_2} + d_2 \delta_{d_2 < d_1} \big)
\tx{.}
\end{gather}

We examine the inner sum more closely. For each $q \in Q$, the congruence condition on $d_2$ tells us that $d_2 \equiv \beta - \td\beta_Q \pmod{a_q} \equiv 0 \pmod{a_q}$. Hence $a_Q \isdiv d_2$. On the other hand, for each $q \in Q_a \backslash Q$, the congruence condition on~$d_1$ implies that~$d_1 \equiv 0 \pmod{q_a}$. Hence $a_Q^\# \isdiv d_1$. 

Recall that~$a' = \gcd(a, 2 \beta)$ and hence~$a'$ must be a divisor of both~$d_1$ and~$d_2$. This forces the divisibility requirements~$a_Q^\# a'_Q \isdiv d_1$ and~$a_Q a^{\prime\,\#}_Q \isdiv d_2$.

We first consider the case~$2 \beta \not\equiv a' \,\pmod{a}$. We restrict $n$ as in the statement of the proposition, by fixing a prime~$p > a \slash a'$ with~$a' p \equiv 2 \beta \,\pmod{a}$ and a further prime~$p' > p a \slash a'$ with~$p' \equiv 1 \,\pmod{a}$. Our aim is to calculate~$c(a' p)$ and~$c(a' p p')$.  The observations in the previous paragraph imply that any $d_1, d_2$ appearing in the sum in \eqref{eq:coefficient-Q-factorizaton} are of the form~$d_1 = a^\#_Q a'_Q k_1$ and~$d_2 = a_Q a^{\prime\,\#}_Q k_2$, with $k_1 k_2 = p$ if $n = a'p$ and $k_1 k_2 = p p'$ if $n=a' p p'$.

We must determine which subsets~$Q \subseteq Q_a$ contribute to the sum in~\eqref{eq:coefficient-Q-factorizaton}. If we have~$n = a'p$, note that~$Q = \emptyset$ indeed yields a factorization~$d_1 = a$, $d_2 = a' p$ that satisfies the given congruence conditions by the assumptions on~$p$. Since~$a' p > a$, its contribution~$-a \slash 2$ to~$c(n) \,\pmod{\ell}$ vanishes by the main theorem of~\cite{beckwith-raum-richter-2020}. Similarly $Q = Q_a$ yields the same contribution coming from the factorization $d_1 = a' p$, $d_2 = a$.

For $n = a' p p'$, we have two factorizations for~$Q = \emptyset$ and~$Q = Q_a$ each that appear. The factorizations~$d_1 = a' p$, $d_2 = a p'$
and~$d_1 = a p$, $d_2 = a' p'$ associated with~$Q = \emptyset$
contribute~$- a' p \slash 2$
and~$- a p \slash 2$
to the sum, since~$p' > p a \slash a'$.
For $Q = Q_a$, the factorizations~$d_1 = a p'$, $d_2 = a' p$ and~$d_1 = a' p'$, $d_2 = a p$ give the same contribution.

We claim that no other $Q$ contributes to the sum. To show this, we employ the prime~$q_a \isdiv a$ with~$a < q_a^2$ and~$2 \beta < q_a$ whose existence is asserted by Condition~\ref{it:la:subprogression:large-prime-divides-a} of Lemma~\ref{la:subprogression}. 

In the case that~$q_a \not\in Q$ we have~$\td\beta_Q \equiv -\beta \,\pmod{q_a}$. We examine the condition
\begin{gather*}
  d_2 \equiv 
  \beta - \td\beta
\equiv
  2 \beta
  \;\pmod{q_a}
\tx{.}
\end{gather*}
Since~$p' \equiv 1 \,\pmod{q_a}$, we know that $d_2 \equiv a_Q a_Q^{\prime\,\#} p \,\pmod{q_a}$ or $d_2 \equiv a_Q a_Q^{\prime\,\#} \,\pmod{q_a}$. In the first case, since~$a'_Q a_Q^{\prime\,\#} p = a' p \equiv 2 \beta \,\pmod{q_a}$ by our assumptions on~$p$ and since~$2 \beta < q_a$ is a unit modulo~$q_a$, this implies the congruence~$a_Q \equiv a'_Q \,\pmod{q_a}$. Since further~$a_Q \slash a'_Q \le a_Q \le a \slash q_a < q_a$, we find that~$a_Q = a'_Q$, and hence~$Q = \emptyset$ by Condition~\ref{it:la:subprogression:proper-local-divisors} of Lemma~\ref{la:subprogression}.

Similarly, in the second case~$a_Q \equiv a_Q' p \,\pmod{q_a}$, that is, $a_Q \slash a_Q' \equiv p \,\pmod{q_a}$.  We have~$a' p \equiv 2 \beta \,\pmod{a}$, and since~$q_a \nisdiv a'$, we find~$p \equiv 2 \beta \slash a' \,\pmod{q_a}$. This yields the congruence~$a_Q \slash a_Q' \equiv 2 \beta \slash a' \,\pmod{q_a}$. Since~$a_Q \slash a_Q' < q_a$ as in the first case and further~$2 \beta \slash a' < 2 \beta < q_a$, we can strengthen it to the equality~$a_Q \slash a_Q' = 2 \beta \slash a'$. We conclude that~$a_Q \slash a_Q'$ is co-prime to~$a$, which by Condition~\ref{it:la:subprogression:proper-local-divisors} of Lemma~\ref{la:subprogression} implies that~$Q = \emptyset$ (hence~$2 \beta = a'$, which cannot occur in the present case~$2 \beta \not\equiv a' \,\pmod{a}$).

Assuming on the other hand that~$q_a \in Q$, we inspect the condition
\begin{gather*}
  d_1 \equiv
  \td\beta + \beta
\equiv
  2 \beta
  \;\pmod{q_a}
\tx{.}
\end{gather*}
Suppose that~$d_1 = a_Q^{\#} a'_Q$. Since~$0 < a^\#_Q a'_Q \le a \slash q_a < q_a$ and~$0 \le 2 \beta < q_a$, we infer that~$a^\#_Q a'_Q = 2 \beta$. By definition, we have~$a^{\prime\,\#}_Q a'_Q = a' = \gcd(2 \beta, a)$, and we conclude that~$a_Q^\# = a^{\prime\,\#}_Q$. Since, however,~$a_q'$ is a proper divisor of~$a_q$ for every prime~$q \isdiv a$ by Condition~\ref{it:la:subprogression:proper-local-divisors} of Lemma~\ref{la:subprogression}, this is only possible if~$a_Q^\# = 1$, and hence~$Q = Q_a$. Similarly, if $d_1 = a_Q^{\#} a'_Q p$, we obtain $a_Q^{\#} a'_Q \equiv 2 \beta \slash p \equiv a' \pmod{q_a}$, and by the same logic we obtain $a_Q^{\#} a'_Q = a'$, implying that $Q = Q_a$.

Summarizing our discussion, for~$2 \beta \not\equiv a' \,\pmod{a}$, we have
\begin{align*}
  c(a' p)
\;&{}=\;
  - a
\;\equiv\;
  0
  \;\pmod{\ell}
\tx{,}
\\
  c(a' p p')
\;&{}=\;
  - (a' + a p)
\;\equiv\;
  - a'
  \;\pmod{\ell}
\tx{.}
\end{align*}

The case~$2 \beta \equiv a' \,\pmod{a}$ is a bit simpler, since we do not need~$p$ as in the above discussion. We can adopt the previous argument to see that the only contributions to the sum over~$Q \subset Q_a$ arise from~$Q = \emptyset$ and~$Q = Q_a$. For~$n = a'$ there is one factorization each associated with~$Q = \emptyset$ and~$Q = Q_a$.  Specifically, the factorizations~$d_1 = a$, $d_2 = a'$ and~$d_1 = a'$, $d_2 = a$ contribute $-a' \slash 2$ each. For~$n = a' p'$, we have two factorizations each associated with~$Q = \emptyset$ and~$Q = Q_a$. The contributions of ~$d_1 = a p'$, $d_2 = a'$ and~$d_1 = a'$, $d_2 = a p'$ equal~$- a' \slash 2$, and the those of ~$d_1 = a' p'$, $d_2 = a$ and~$d_1 = a$, $d_2 = a' p'$ equals~$-a \slash 2$. In summary, we find that
\begin{align*}
  c(a')
\;&{}=\;
  -a'
\tx{,}
\\
  c(a' p')
\;&{}=\;
  - (a' + a)
\;\equiv\;
  - a'
  \;\pmod{\ell}
\tx{,}
\end{align*}
where in the last congruence we again have invoked the main theorem of~\cite{beckwith-raum-richter-2020}.
\end{proof}

\begin{proof}[Proof of Theorem~\ref{mainthm:nonholomorphic-ell-divides-b}]
We establish the theorem by contraposition. Assume that~$\ell \nisdiv b$. Lemma~\ref{la:subprogression} allows us to replace~$a$ and~$b$ in such a way that we can apply Proposition~\ref{prop:coefficient-formula}. 

We can now proceed as in~\cite{beckwith-raum-richter-2020} and apply Theorem \ref{thm:serre} to deduce a contradiction. Let $a' = (2\beta, a)$. By assumption, $\ell \nmid a'$. First suppose $a' \not\equiv 2\beta \,\pmod{\ell}$. We choose $p > a \slash a' $ such that $a' p \equiv 2\beta \,\pmod{a}$. Then we choose $p' > p a \slash a'$ as in Theorem~\ref{thm:serre}. The first part of Proposition~\ref{prop:coefficient-formula} says that we must have
\begin{gather*}
  c(a' p) \;\equiv\; 0 \;\pmod{\ell}
\quad\tx{and}\quad
  c(a' p p') \;\equiv\; -a' \not\equiv 0 \;\pmod{\ell}
\tx{,}
\end{gather*}
but Theorem~\ref{thm:serre} leads to
\begin{gather*}
  c(a' p p') \;\equiv\; 2 c(a' p') \;\pmod{\ell}
\tx{,}
\end{gather*}
and hence the contradiction~$c(a' p p') \equiv 0 \,\pmod{\ell}$.

Now suppose~$a' \equiv 2\beta \,\pmod{a}$. Let~$p' > a \slash a'$ be as in Theorem~\ref{thm:serre}. Then
\begin{gather*}
  c(a')
\;\equiv\;
  c(a' p')
\;\equiv\;
  -a'
\;\not\equiv\;
  0 \;\pmod{\ell}
\end{gather*}
by the second part of Proposition~\ref{prop:coefficient-formula}, but by Theorem~\ref{thm:serre} we have
\begin{gather*}
  2 c(a') \;\equiv\; c(a' p) \;\pmod{\ell}
\tx{.}
\end{gather*}
Hence~$c(a') \equiv 0 \,\pmod{\ell}$, a contradiction.
\end{proof}

\section{Congruences on square-classes}%
\label{sec:square-classes}

%\begin{theorem}
%\label{thm:square-classes}
%Fix a prime~$\ell > 3$. Let~$a > 0$ and~$b$ be integers. Suppose~$H(an + b) \equiv 0 \,\pmod{\ell}$ for all $n$. Then $H(an + b h^2) %\equiv 0 \,\pmod{\ell}$ for all integers~$h$ with $\gcd(h,a)=1$ and $n \in \ZZ$. 
%\end{theorem}

The proof of Theorem~\ref{mainthm:square-classes} is split into two parts. Both require the following lemma.
\begin{lemma}
\label{la:modular-transformation-on_Uab}
Let
\begin{gather*}
  \ga = \begin{psmatrix} a_\ga & b_\ga \\ c_\ga & d_\ga \end{psmatrix} \in \Ga_0(4a)
\quad\tx{satisfy}\quad
  \ga \equiv \begin{psmatrix} h & 0 \\ 0 & \ov{h}\end{psmatrix} \;\pmod{4 a}
\tx{,}
\end{gather*}
where $\ov{h} \,\pmod{4 a}$ is a multiplicative inverse of\/~$h$ modulo~$4a$, and assume that $c_\ga \slash 4a$ is relatively prime to $2a$. Then there is~$\omega(\ga) \in \{ \pm 1, \pm i \}$ such that
\begin{gather*}
  \rmU_{a,b}\, E_{\frac{3}{2}}
  \big|_{\frac{3}{2}}\,\ga
\;=\;
  \omega(\ga)\,
  \rmU_{a, b h^2}\, E_{\frac{3}{2}}
\quad\tx{and}\quad
  \rmU_{a,b}\, \theta
  \big|_{\frac{1}{2}}\,\ga
\;=\;
  \ov{\omega(\ga)}\,
  \rmU_{a, b h^2}\, \theta
\tx{.}
\end{gather*}
\end{lemma}
\begin{proof}
We give the argument only in the case of the Eisenstein series. The case of the theta series follows form almost literally the same calculation. We can write $\rmU_{a,b}$ as a double coset operator:
\begin{gather*}
  \rmU_{a,b}\,E_{\frac{3}{2}}
\;=\;
  a^{\frac{3}{4} - 1}
  \sum_{\lambda \,\pmod{a}}
  e\big( \mfrac{- \lambda b}{a} \big)\,
  E_{\frac{3}{2}} \big|_{\frac{3}{2}}\,
  \begin{psmatrix} 1 & \lambda \\ 0 & a \end{psmatrix}
\tx{.}
\end{gather*}
We calculate
\begin{gather*}
  \begin{pmatrix} 1 & \lambda \\ 0 & a \end{pmatrix}
  \ga
\;=\;	
  \begin{pmatrix}
    a_\ga + c_\ga \lambda
    &
    \frac{1}{a} (-\lambda \ov{h}{}^2 a_\ga + b_\ga ) +
    \frac{\lambda}{a}(-c_\ga \lambda \ov{h}{}^2 + d_\ga)
    \\
    a c_\ga & d_\ga - \ov{h}{}^2 \lambda c_\ga
  \end{pmatrix}
  \begin{pmatrix} 1 & \lambda \ov{h}{}^2 \\ 0 & a \end{pmatrix}
\end{gather*}
to find that
\begin{align*}
  \big( \rmU_{a,b}\,E_{\frac{3}{2}} \big) \big|_{\frac{3}{2}}\,\ga
\;&{}=\;
  a^{\frac{3}{4} - 1}
  \sum_{\lambda \,\pmod{a}}
  e\big( \mfrac{- \lambda b}{a} \big)\,
  E_{\frac{3}{2}} \big|_{\frac{3}{2}}\,
  \begin{psmatrix} 1 & \lambda \\ 0 & a \end{psmatrix}
  \ga
\\
&{}=\;
  a^{\frac{3}{4} - 1}
  \sum_{\lambda \,\pmod{a}}
  e\big( \mfrac{- (\lambda \ov{h}{}^2) (b h^2)}{a} \big)\,
  \epsilon_{d_\ga - \lambda \ov{h}{}^2  c_\ga}\,
  \left( \mfrac{a c_\ga }{d_\ga - \lambda \ov{h}{}^2 c_\ga} \right)\,
  E_{\frac{3}{2}} \big|_{\frac{3}{2}}\,
  \begin{psmatrix} 1 & \lambda \ov{h}{}^2 \\ 0 & a \end{psmatrix}
\tx{.}
\end{align*}

Since~$c_\ga$ is divisible by $4 a$, we have $d_\ga - \lambda \ov{h}{}^2 c_\ga \equiv d_\ga \,\pmod{4}$, hence
\begin{gather*}
  \epsilon_{d_\ga - \lambda \ov{h}{}^2  c_\ga}
\;=\;
  \epsilon_{d_\ga}
\tx{.}
\end{gather*}
Let $c'_\ga : = c_\ga \slash 4a$. Using quadratic reciprocity, we obtain
\begin{gather*}
  \left( \mfrac{a c_\ga}{d_\ga - \lambda \ov{h}{}^2 c_\ga} \right)
=
  \left( \mfrac{c'_\ga}{d_\ga - \lambda \ov{h}{}^2 c_\ga} \right)
=
  \left( \mfrac{c'_\ga}{d_\ga} \right)
=
  \left( \mfrac{a c_\ga}{d_\ga} \right)
\tx{.}
\end{gather*}

Let
\begin{gather*}
  \omega_\ga
\;:=\;
  \epsilon_{d_\ga}
  \left( \mfrac{a c_\ga}{d_\ga} \right)
\,\in\,
  \{\pm 1, \pm i\}
\tx{.}
\end{gather*}
Then we have
\begin{align*}
  \big( U_{a,b}\, E_{\frac{3}{2}} \big) \big|_{\frac{3}{2}} \ga
&\;{}=\;
  a^{\frac{3}{4} - 1} \omega(\ga)
  \sum_{\lambda \,\pmod{a}}
  e\left( \mfrac{- (\lambda \ov{h}{}^2) (b h^2)}{a} \right)\,
  E_{\frac{3}{2}} \big|_{\frac{3}{2}}
   \begin{psmatrix} 1 & \lambda \ov{h}{}^2 \\ 0 & a \end{psmatrix}
\\
&=
  \omega (\ga)\,
  U_{a,b h^2}\, E_{\frac{3}{2}} 
\tx{.}
\end{align*}
\end{proof}

\begin{proof}%
[Part~1 of the proof of Theorem~\ref{mainthm:square-classes}]
We prove the theorem in the case that~$-b$ is not a square modulo~$a$, which implies that
\begin{gather*}
  0
\;\equiv\;
  \sum_{n \in \ZZ} H(an + b) e\big((an+b) \tau \big)
\;=\;
  \rmU_{a,b}\,E_{\frac{3}{2}}
\in
  \rmM_{\frac{3}{2}} \big( \Gamma(4a), \nu_{\theta}^3 \big)
\tx{.}
\end{gather*}

Fix~$u$ as in the statement. By replacing~$u$ by~$u + a$, if needed, we can and will assume that~$\gcd(u, 2a) = 1$. Let $\ga \in \Ga_0(4 a \ell)$ satisfy $\ga \equiv \begin{psmatrix} u & 0 \\ 0 & \ov{u}\end{psmatrix} \,\pmod{4 a}$ where $\ov{h} \,\pmod{4 a}$ is a multiplicative inverse of~$u$ modulo~$4a$, and assume that $c/4a$ is relatively prime to $2a$. We combine the $q$-expansion principle (see Lemma 2.3, \cite{ahlgren-beckwith-raum-2020-preprint}) with Lemma~\ref{la:modular-transformation-on_Uab} to find that
\begin{gather}
\label{eq:thm:square-classes:proof-part1-fourier-expansion}
  0
\;\equiv\;
  \big( U_{a,b}\, E_{\frac{3}{2}} \big) \big|_{\frac{3}{2}} \ga
=
  \omega (\ga)\,
  U_{a,b u^2}\, E_{\frac{3}{2}} 
  \;\pmod{\ell}
\tx{,}
\end{gather}
where the congruence is to be understood in the ring of Gaussian integers if~$\omega(\ga)$ does not lie in~$\QQ$. We obtain the statement from the Fourier expansion of the right hand side of~\eqref{eq:thm:square-classes:proof-part1-fourier-expansion}.
\end{proof}

The second part of our proof of Theorem~\ref{mainthm:square-classes} requires two further lemmas.
\begin{lemma}
\label{la:hol-proj-theta-product-vanishes}
Fix a prime~$\ell > 3$. Let~$a > 0$ and~$\beta, \beta'$ be integers that are divisible by~$\ell$. Then we have
\begin{gather*}
  \mfrac{\sqrt{a}}{\pi}
  \pi^\hol_2 \big( 
  \theta_{a,\td\beta}^\ast
  \cdot
  \theta_{a,\beta}
  \big)
\;\equiv\;
  0
  \;\pmod{\ell}
\tx{.}
\end{gather*}

In particular, for integers~$b, b'$ that are divisible by~$\ell$, we have
\begin{gather*}
  \mfrac{1}{\pi}
  \pi^\hol_2 \big( 
  \rmU_{a,b'}\, \theta^\ast
  \cdot
  \rmU_{a,b}\, \theta
  \big)
\;\equiv\;
  0
  \;\pmod{\ell}
\end{gather*}
and
\begin{gather*}
  \pi^\hol_2 \big( 
  \rmU_{a,b}\, E_{\frac{3}{2}}
  \cdot
  \rmU_{a,b'}\, \theta
  \big)
\equiv
  \rmU_{a,b}\, E^\hol_{\frac{3}{2}}
  \cdot
  \rmU_{a,b'}\, \theta
  \;\pmod{\ell}
\tx{.}
\end{gather*}
\end{lemma}
\begin{proof}
The second part follows from the first part, in light of~\eqref{eq:Uab-theta}. The third part follows from the second one and~\eqref{eq:zagier-eisenstein-series}:
\begin{multline*}
  \pi^\hol_2 \big( 
  \rmU_{a,b}\, E_{\frac{3}{2}}
  \cdot
  \rmU_{a,b'}\, \theta
  \big)
=
  \pi^\hol_2 \big( 
  \rmU_{a,b}\, \big( E^\hol_{\frac{3}{2}} + \mfrac{1}{16 \pi} \theta^\ast \big)
  \cdot
  \rmU_{a,b'}\, \theta
  \big)
\\
=
  \pi^\hol_2 \big( 
  \rmU_{a,b}\, E^\hol_{\frac{3}{2}}
  \cdot
  \rmU_{a,b'}\, \theta
  \big)
  \,+\,
  \mfrac{1}{16 \pi}
  \pi^\hol_2 \big( 
  \rmU_{a,b}\, \theta^\ast
  \cdot
  \rmU_{a,b'}\, \theta
  \big)
=
  \rmU_{a,b}\, E^\hol_{\frac{3}{2}}
  \cdot
  \rmU_{a,b'}\, \theta
\tx{.}
\end{multline*}

We employ the calculations from the proof of Proposition~2.5 in~\cite{beckwith-raum-richter-2020}, in the same way as we used it in the proof of Proposition~\ref{prop:coefficient-formula}, to establish the first congruence. We have (compare Equation~[16] of~\cite{beckwith-raum-richter-2020})
\begin{multline*}
  \mfrac{\sqrt{a}}{\pi}
  \pi^\hol_2\big(
  \theta^\ast_{a,\td\beta}(\tau) \cdot \theta_{a,\beta}(\tau)
  \big)
\\
=\;
  -4
  \Big(
  \delta_{\td\beta \equiv 0 \,\pmod{a}}
  \sum_{\substack{m \equiv \beta \,\pmod{a}\\m \ne 0}}
  |m| e\big( \mfrac{m^2 \tau}{a} \big)
  \,+\,
  \sum_{\substack{m \equiv \beta \,\pmod{a}\\\td{m} \equiv \td\beta \,\pmod{a}\\\td{m} \ne 0}}
  \frac{m^2 - \td{m}^2}{|m| + |\td{m}|}
  e\Big( \frac{(m^2 - \td{m}^2) \tau}{a} \Big)
  \Big)
\tx{.}
\end{multline*}
Since~$\ell \isdiv a, \beta$, the Fourier coefficients in the first summand are all divisible by~$\ell$. Consider a term in the second sum for fixed~$m$ and~$\td{m}$. The ratio in this term is divisible by either~$m + \td{m}$ or~$m - \td{m}$. Since~$\ell \isdiv m, \td{m}$, this proves the lemma.
\end{proof}

\begin{lemma}
\label{la:quasi-modular-with-fractional-fourier-indices}
Let~$N$ be a positive integer and~$f$ be a quasi-modular form of weight~$2$ with Fourier expansion
\begin{gather*}
  f(\tau)
\;=\;
  \sum_{\substack{n = 0\\N \nisdiv n}}^\infty
  c\big(f;\tfrac{n}{N}\big) e\big(\tfrac{n}{N} \tau \big)
\tx{.}
\end{gather*}
Then~$f$ is a modular form.
\end{lemma}
\begin{proof}
We decompose~$f$ as a sum~$c E_2^{\mathrm{hol}} + g$ for a constant~$c \in \CC$ and a modular form~$g$ of level~$N$ (in the sense of Wohlfahrt~\cite{wohlfahrt-1964}), where~$E_2^{\mathrm{hol}}$ is the quasi-modular, holomorphic part of the weight~$2$ Eisenstein series. From the Fourier expansion of~$f$, we infer that
\begin{gather*}
  0
\;=\;
  \mfrac{1}{N} \sum_{m = 1}^N f \big|_2 \begin{psmatrix} 1 & m \\ 0 & 1 \end{psmatrix}
=
  c E_2^{\mathrm{hol}}
  \,+\,
  \mfrac{1}{N} \sum_{m = 1}^N g \big|_2 \begin{psmatrix} 1 & m \\ 0 & 1 \end{psmatrix}
\tx{.}
\end{gather*}
Since the second summand is a modular form, we conclude that~$c = 0$ as desired.
\end{proof}

\begin{proof}%
[Part~2 of the proof of Theorem~\ref{mainthm:square-classes}]
We now prove the theorem in the case that~$-b$ is a square modulo~$a$. We start with the congruences
\begin{gather*}
  \rmU_{a,b}\, E^\hol_{\frac{3}{2}}
  \cdot
  \rmU_{a,\td{b}}\, \theta
\;\equiv\;
  0
  \;\pmod{\ell}
\tx{,}
\end{gather*}
which hold for all integer~$\td{b}$, since the first factor vanishes modulo~$\ell$ by our assumptions. The main theorem of~\cite{beckwith-raum-richter-2020} informs us that~$\ell \isdiv a$, and Theorem~\ref{mainthm:nonholomorphic-ell-divides-b} asserts that~$\ell \isdiv b$. If~$a \isdiv b$, there is nothing to show. We assume the opposite. Then there is some integer~$b'$ with~$\ell \isdiv b'$ and~$-b \not\equiv b' \,\pmod{a}$ and~$b'$ is a square modulo~$a$. Lemma~\ref{la:hol-proj-theta-product-vanishes} now yields the following congruence of quasi-modular forms:
\begin{gather}
\label{eq:thm:square-classes-nonholomorphic:hol-proj-product}
  \pi^\hol_2 \big( 
  \rmU_{a,b}\, E_{\frac{3}{2}}
  \cdot
  \rmU_{a,b'}\, \theta
  \big)
\;\equiv\;
  0
  \;\pmod{\ell}
\tx{.}
\end{gather}

Since~$b+b' \not\equiv 0 \,\pmod{a}$, we can apply Lemma~\ref{la:quasi-modular-with-fractional-fourier-indices} to the left hand side of~\eqref{eq:thm:square-classes-nonholomorphic:hol-proj-product} to see that it is a modular form as opposed to only a quasi-modular form. In particular, we can apply the Fourier expansion principle (in its weak form) due to Katz~\cite{katz-1973} (see also the strong form by Deligne-Rapoport~\cite{deligne-rapoport-1973}). Similar to our strategy in the first of part of this proof, let $\ga \in \Ga(4a)$ satisfy $\ga \equiv \begin{psmatrix} u & 0 \\ 0 & \ov{u}\end{psmatrix} \,\pmod{4a}$ where $\ov{u} \,\pmod{4a}$ is a multiplicative inverse of~$u$ modulo~$4a$, and assume that $c_\ga \slash 4a$ is relatively prime to $4a$, where~$c_\ga$ is the bottom left entry of~$\ga$. The Fourier expansion principle yields the congruence
\begin{gather*}
  \pi^\hol_2 \big( 
  \rmU_{a,b}\, E_{\frac{3}{2}}
  \cdot
  \rmU_{a,b'}\, \theta
  \big)
  \big|_2\,\ga
\;\equiv\;
  0
  \;\pmod{\ell}
\tx{.}
\end{gather*}

To determine the left hand side of this congruence, recall that the slash action intertwines with the holomorphic projection (see~\cite{imamoglu-raum-richter-2014}). We have
\begin{gather*}
  \pi^\hol_2 \Big( 
  \big(
  \rmU_{a,b}\, E_{\frac{3}{2}}
  \big|_{\frac{3}{2}}\,\ga
  \big)
  \cdot
  \big(
  \rmU_{a,b'}\, \theta
  \big|_{\frac{1}{2}}\,\ga
  \big)
  \Big)
\;\equiv\;
  0
  \;\pmod{\ell}
\tx{.}
\end{gather*}
Lemma~\ref{la:modular-transformation-on_Uab} yields the congruence
\begin{gather*}
  \pi^\hol_2 \big( 
  \rmU_{a, b u^2}\, E_{\frac{3}{2}}
  \cdot
  \rmU_{a, b' u^2}\, \theta
  \big)
\;\equiv\;
  0
  \;\pmod{\ell}
\tx{.}
\end{gather*}
Observe that~$b u^2$ and~$b' u^2$ are divisible by~$\ell$, so that we can apply Lemma~\ref{la:hol-proj-theta-product-vanishes} to find that
\begin{gather}
\label{eq:thm:square-classes-nonholomorphic:final-product}
  \rmU_{a, b u^2}\, E^\hol_{\frac{3}{2}}
  \cdot
  \rmU_{a, b' u^2}\, \theta
\;\equiv\;
  0
  \;\pmod{\ell}
\tx{.}
\end{gather}
Since~$b'$ is a square modulo~$a$ by assumption, the second factor on the left hand side of~\eqref{eq:thm:square-classes-nonholomorphic:final-product} is not congruent to zero modulo~$\ell$, and more specifically its first non-zero Fourier coefficient equals one or two. From this we infer that $\rmU_{a, b u^2}\, E^\hol_{\frac{3}{2}} \equiv 0 \,\pmod{\ell}$.
\end{proof}

\section{Proofs of Theorems~\ref{mainthm:a-b-ppowers} and~\ref{mainthm:dichotomy}}
\label{sec:applications}

\begin{proof}[Proof of Theorem~\ref{mainthm:a-b-ppowers}]
For simplicity, we say that~$(a,b)$ is a mod~$\ell$ Hurwitz congruence pair if we have the Ramanujan-type congruence $H(an + b) \equiv 0 \,\pmod{\ell}$ for all~$n \in \ZZ$. Furthermore, we say that a mod~$\ell$ Hurwitz congruence pair~$(a,b)$ is maximal if the corresponding Ramanujan-type congruence is maximal.

Let~$p$ be a prime, let~$k = \ord_p(\gcd(a,b))$ and~$r = \ord_p (a \slash \gcd(a,b))$. After replacing~$b$ by~$b + a$, if needed, we then have $a = p^{k+r} a'$ and $b = p^k b'$ for integers~$a'$ and~$b'$ with $\gcd(a' b' ,p) = 1$. 

First, we assume that $p$ is odd, that $r \ge 2$, and that $(a,b)$ is a Hurwitz congruence pair modulo $\ell$. We will show that $(a,b)$ is not a maximal Hurwitz congruence pair mod $\ell$. If $m \equiv b \,\pmod{a \slash p}$, then $p^k \| m$ and $m/p^k \equiv b' \,\pmod{p^{r-1}}$. From Hensel's Lemma, there exists~$u \in \ZZ$ with~$\gcd(u,p)= 1$ such that~$m \slash p^k \equiv b' u^2 \pmod{p^r}$. Using the Chinese Remainder Theorem, one can find such a~$u$ with $u \equiv 1 \pmod{a'}$ so that we have $m \equiv b u^2 \,\pmod{a}$. By Theorem~\ref{mainthm:square-classes}, we have $H(m) \equiv 0 \pmod{\ell}$. Hence $(a \slash p, b)$ is a Hurwitz congruence pair modulo~$\ell$, so~$(a,b)$ is not a maximal Hurwitz congruence pair.

The $p=2$ case is almost exactly the same. We assume~$r \ge 4$  and we will show that $(a,b)$ cannot be a maximal Hurwitz congruence pair modulo~$\ell$. Suppose $m \equiv b \,\pmod{a \slash 2}$. Then $2^k | m$, and using a Hensel's lemma type argument, one easily checks that there exists an integer~$u$ which is relatively prime to $a$ such that $m \equiv b u^2 \pmod{a}$ (this is where we require $r \ge 4$ rather than $r \ge 2$). By the Theorem, we have $H(m) \equiv 0\pmod{\ell}$, which means $(a,b)$ is not a maximal Hurwitz congruence pair.
\end{proof} 

\begin{proof}[Proof of Theorem~\ref{mainthm:dichotomy}]
Assume that~\ref{it:mainthm:dichotomy:fundamental-discriminants} of Theorem~\ref{mainthm:dichotomy} does not hold. That is, there is a fundamental discriminant~$-D$ and a positive integer~$f$ such that~$D f^2 \in a \ZZ + b$ and~$H(D) \not\equiv 0 \,\pmod{\ell}$. Given a prime~$p \isdiv f$ we write~$f_p$ for its~$p$\nbd part. We will show that there is a prime~$p \isdiv \gcd(f,a)$ such that
\begin{gather}
\label{eq:mainthm:dichotomy:proof_congruence}
  \sigma_1(f_p) - \sigma_1\Big( \mfrac{f_p}{p} \Big) \Big(\mfrac{-D}{p}\Big)
\equiv
  0
  \;\pmod{\ell}
\tx{.}
\end{gather}

We factor~$f = f_a u$ into the product of two positive integers~$f_a$ and~$u$, where every prime dividing~$f_a$ also divides~$a$ and~$u$ is co-prime to~$a$. In particular, there is an inverse~$\ov{u}$ of~$u$ modulo~$a$.

Theorem~\ref{mainthm:square-classes} asserts that we have a Ramanujan-type congruence modulo~$\ell$ for Hurwitz class numbers on~$a \ZZ + \ov{u}^2 b \ni D f_a^2$. The Hurwitz class number formula asserts that
\begin{align*}
  H(D f_a^2)
&{}=
  H(D)\,
  \frac{\omega(-D f_a^2)}{\omega(-D)}
  \sum_{d \isdiv f_a}
  d
  \prod_{p \isdiv d}
  \Big( 1 - \mfrac{1}{p} \Big(\mfrac{-D}{p}\Big) \Big)
\\
&{}=
  H(D)\,
  \frac{\omega(-D f_a^2)}{\omega(-D)}
  \prod_{p \isdiv f_a}
  \Big( \sigma_1(f_p) -  \sigma_1\Big( \mfrac{f_p}{p} \Big) \Big(\mfrac{-D}{p}\Big) \Big)
\tx{,}
\end{align*}
where~$\omega(-D f_a^2)$ is the number of units in the imaginary quadratic order of discriminant~$-D f_a^2$. By assumption on~$D$, we have~$H(D) \not\equiv 0 \,\pmod{\ell}$. Further, we note that~$\omega(-D f_a^2) \not\equiv 0 \,\pmod{\ell}$, since~$\ell > 3$ and~$\omega(-D f_a^2) \isdiv 6$. Therefore, we conclude the existence of some prime~$p \isdiv \gcd(f,a)$ satisfying~\eqref{eq:mainthm:dichotomy:proof_congruence} as desired.

 To finish the proof, we will show that for any other fundamental discriminant~$-D'$ and integer~$f'$ with~$D' f'^2 \in a \ZZ+b$, we have~$f_p = f_p'$.  We note that $( \frac{-D'}{p} ) = ( \frac{-D}{p} )$ follows from $D' f'^2, Df^2 \in a \mathbb{Z} + b$ and $f_{2}' = f_{2}$. Case \ref{it:mainthm:dichotomy:hecke-eigenvalues} follows immediately from these two claims and (\ref{eq:mainthm:dichotomy:proof_congruence}). 
 
 By assumption~$\ord_p(a \slash \gcd(a,b)) > 0$, and hence~$a_p \nisdiv b$ and~$\ord_p(b) = \ord_p(D f^2) = \ord_p(D' f'^2)$. In other words, we have
\begin{gather}\label{eq:ordpformula}
  \ord_p(D') + 2 \ord_p(f')
=
  \ord_p(b)
=
  \ord_p(D) + 2 \ord_p(f)
\tx{.}
\end{gather}

Now consider the case of odd~$p$. Since~$-D'$ and~$-D$ are fundamental discriminants, $\ord_p(D') \le 1$ and~$\ord_p(D) \le 1$ are given by the parity of~$\ord_p(b)$ and hence~$f_p = f_p'$ as required. 
 
Next, consider the case of~$p = 2$. We have~$\ord_2(D) \in \{0, 2, 3\}$. If~$\ord_2(D) = 3$ or $\ord_2(D')=3$, the argument for odd~$p$ extends.  From now on, we assume that $\ord_2(D)$ and $\ord_2(D')$ are both in $\{0,2\}$. If~$\ord_2(D) = 0$ and $\ord_2(D') = 2$, then from (\ref{eq:ordpformula}) we have
 \begin{gather*}
   2 + 2 \ord_p(f')
=
  \ord_2(b)
=
  2 \ord_p(f).
\end{gather*}
From $D'f'^2, Df^2 \in a \ZZ +b$ and $\ord_2(a \slash \gcd(a,b) ) \ge 2$, we obtain 
\begin{gather*}
  \frac{D' f'^2}{4 f_2'^2} \equiv \frac{D}{f_2^2} \equiv 0 \;\pmod{4}
\tx{,}
\end{gather*}
from which we have~$D \equiv D' \slash 4 \,\pmod{4}$. Since $-D'$ is a fundamental discriminant and $\ord_2(D') = 2$, we would have $D' \slash 4 \equiv 1 \pmod{4}$. Therefore $D \equiv 1 \pmod{4}$. This is a contradiction, since $-D$ is a discriminant. The case of~$\ord_2(D) = 2$ and~$\ord_2(D')$ is excluded by a symmetric argument. We conclude that we must have $\ord_2(D') = \ord_2(D)$, which implies $f_2 = f_2'$ as desired.
\end{proof}

%%%%%%%%%%%%%%%%%%%%%%%%%%%%%%%%%%%%%%%%%%%%%%%%%%

%%% BIBLIOGRAPHY

\printbibliography

@book{bringmann-folsom-ono-rolen-2018,
    AUTHOR = {Bringmann, Kathrin and Folsom, Amanda and Ono, Ken and Rolen,
              Larry},
     TITLE = {Harmonic {M}aass forms and mock modular forms: theory and
              applications},
    SERIES = {American Mathematical Society Colloquium Publications},
    VOLUME = {64},
 PUBLISHER = {American Mathematical Society, Providence, RI},
      YEAR = {2017},
     PAGES = {xv+391}, }

@article{radu-2013,
    AUTHOR = {Radu, Cristian-Silviu},
     TITLE = {Proof of a conjecture by {A}hlgren and {O}no on the non-existence of certain partition congruences},
   JOURNAL = {Trans. Amer. Math. Soc.},
  FJOURNAL = {Transactions of the American Mathematical Society},
    VOLUME = {365},
      YEAR = {2013},
    NUMBER = {9},
     PAGES = {4881--4894}, }

@article{serre-1974,
    AUTHOR = {Serre, Jean-Pierre},
     TITLE = {Divisibilit\'{e} des coefficients des formes modulaires de poids entier},
   JOURNAL = {C. R. Acad. Sci. Paris S\'{e}r. A},
    VOLUME = {279},
      YEAR = {1974},
     PAGES = {679--682}, }

@inproceedings{serre-1973,
    AUTHOR = {Serre, Jean-Pierre},
     TITLE = {Formes modulaires et fonctions z\^{e}ta {$p$}-adiques},
 BOOKTITLE = {Modular functions of one variable, {III} ({P}roc. {I}nternat.  {S}ummer {S}chool, {U}niv. {A}ntwerp, 1972)},
     PAGES = {191--268. Lecture Notes in Math., Vol. 350},
      YEAR = {1973}, }

@article{beckwith-raum-richter-2020,
    AUTHOR = {Beckwith, Olivia and Raum, Martin and Richter, Olav K.},
     TITLE = {Nonholomorphic {R}amanujan-type congruences for {H}urwitz class numbers},
   JOURNAL = {Proc. Natl. Acad. Sci. USA},
  FJOURNAL = {Proceedings of the National Academy of Sciences of the United
              States of America},
    VOLUME = {117},
      YEAR = {2020},
    NUMBER = {36},
     PAGES = {21953--21961}, }

@article{zagier-1994,
    AUTHOR = {Zagier, Don},
     TITLE = {Modular forms and differential operators},
      NOTE = {K. G. Ramanathan memorial issue},
   JOURNAL = {Proc. Indian Acad. Sci. Math. Sci.},
  FJOURNAL = {Indian Academy of Sciences. Proceedings. Mathematical Sciences},
    VOLUME = {104},
      YEAR = {1994},
    NUMBER = {1},
     PAGES = {57--75}, }

@incollection{kaneko-zagier-1995,
    AUTHOR = {Kaneko, Masanobu and Zagier, Don},
     TITLE = {A generalized {J}acobi theta function and quasimodular forms},
 BOOKTITLE = {The moduli space of curves ({T}exel {I}sland, 1994)},
    SERIES = {Progr. Math.},
    VOLUME = {129},
     PAGES = {165--172},
 PUBLISHER = {Birkh\"{a}user Boston, Boston, MA},
      YEAR = {1995}, }

@article{wohlfahrt-1964,
 Author = {Wohlfahrt, Klaus},
 Title = {{An extension of F. Klein's level concept}},
 FJournal = {{Illinois Journal of Mathematics}},
 Journal = {{Ill. J. Math.}},
 Volume = {8},
 Pages = {529--535},
 Year = {1964},
 Publisher = {Duke University Press, Durham, NC; University of Illinois, Department of Mathematics, Urbana, IL}, }

@article{zagier-1975,
  AUTHOR =	 {Zagier, Don B.},
  TITLE =	 {Nombres de classes et formes modulaires de poids {$3/2$}},
  JOURNAL =	 {C. R. Acad. Sci. Paris S\'er. A-B},
  VOLUME =	 {281},
  YEAR =	 {1975},
  NUMBER =	 {21},
  PAGES =	 {Ai, A883--A886}, }

@article{cohen-1975,
    AUTHOR = {Cohen, Henri},
     TITLE = {Sums involving the values at negative integers of {$L$}-functions of quadratic characters},
   JOURNAL = {Math. Ann.},
  FJOURNAL = {Mathematische Annalen},
    VOLUME = {217},
      YEAR = {1975},
    NUMBER = {3},
     PAGES = {271--285}, }

@misc{jochnowitz-2004-preprint,
  author = {Jochnowitz, Naomi},
   title = {Congruences between modular forms of half integral weights and implications for class numbers and elliptic curves},
 note = {preprint}, }

@article{sturm-1980,
    AUTHOR = {Sturm, Jacob},
     TITLE = {Projections of {$C^{\infty }$}\ automorphic forms},
   JOURNAL = {Bull. Amer. Math. Soc. (N.S.)},
  FJOURNAL = {American Mathematical Society. Bulletin. New Series},
    VOLUME = {2},
      YEAR = {1980},
    NUMBER = {3},
     PAGES = {435--439}, }

@article{gross-zagier-1986,
    AUTHOR = {Gross, Benedict H. and Zagier, Don B.},
     TITLE = {Heegner points and derivatives of {$L$}-series},
   JOURNAL = {Invent. Math.},
  FJOURNAL = {Inventiones Mathematicae},
    VOLUME = {84},
      YEAR = {1986},
    NUMBER = {2},
     PAGES = {225--320}, }

@article{imamoglu-raum-richter-2014,
    AUTHOR = {Imamo{\u{g}}lu, {\"O}zlem and Raum, Martin and Richter, Olav K.},
     TITLE = {Holomorphic projections and {R}amanujan's mock theta functions},
   JOURNAL = {Proc. Natl. Acad. Sci. USA},
  FJOURNAL = {Proceedings of the National Academy of Sciences of the United States of America},
    VOLUME = {111},
      YEAR = {2014},
    NUMBER = {11},
     PAGES = {3961--3967}, }

@misc{ahlgren-beckwith-raum-2020-preprint, 
    AUTHOR = {Ahlgren, Scott and Beckwith, Olivia and Raum, Martin},
    TITLE = {Scarcity of congruences for the partition function},
    howpublished = {arXiv: 2006.07645},
    year = 2020,}

@incollection{katz-1973,
    AUTHOR = {Katz, Nicholas M.},
     TITLE = {{$p$}-adic properties of modular schemes and modular forms},
  CROSSREF = {modular-forms-one-variable-III},
     PAGES = {69--190}, }

@book{modular-forms-one-variable-III,
     TITLE = {Modular functions of one variable, {III}},
      note = {{P}roc. {I}nternat. {S}ummer {S}chool, {U}niv. {A}ntwerp, {A}ntwerp, (1972)},
    SERIES = {Lecture Notes in Math.},
    volume = {350},
 PUBLISHER = {Springer, Berlin},
      YEAR = {1973}, }

@incollection{deligne-rapoport-1973,
    AUTHOR = {Deligne, Pierre and Rapoport, Michael},
     TITLE = {Les sch\'{e}mas de modules de courbes elliptiques},
  CROSSREF = {modular-forms-one-variable-II},
     PAGES = {143--316},
 }

@book{modular-forms-one-variable-II,
     TITLE = {Modular functions of one variable {II}},
      note = {{P}roc. {I}nternat. {S}ummer {S}chool, {U}niv. {A}ntwerp, {A}ntwerp, (1971)},
    SERIES = {Lecture Notes in Math.},
    volume = {349},
 PUBLISHER = {Springer, Berlin},
      YEAR = {1973}, }

@article{Wiles,
author = {Wiles, A.},
title = {On class groups of imaginary quadratic fields},
journal = {Journal of the London Mathematical Society},
volume = {92},
number = {2},
pages = {411-426},
doi = {https://doi.org/10.1112/jlms/jdv031},
year = {2015}
}

%\ifbool{nobiblatex}{%
  %\bibliographystyle{alpha}%
  %\bibliography{bibliography.bib}%
%}{%
%  \renewcommand{\baselinestretch}{.8}
 % \Needspace*{4em}
%  \begin{multicols}{2}
 % \printbibliography[heading=none]%[heading=bibnumbered]
%  \end{multicols}

%\bibliographystyle{amsalpha}
%\bibliographystyle{alphanum}
%\bibliography{bibliography}
%\addbibresource{bibliography.bib}
%%%%%%%%%%%%%%%%%%%%%%%%%%%%%%%%%%%%%%%%%%%%%%%%%%
%%% AFFILIATIONS

\Needspace*{3\baselineskip}
\noindent
\rule{\textwidth}{0.15em}

{\noindent\small
Olivia Beckwith\\
Mathematics Department,
Tulane University,
New Orleans, LA 70118, USA\\
E-mail: \url{obeckwith@tulane.edu}\\
Homepage: \url{https://www.olivia-beckwith.com/}
}\vspace{.5\baselineskip}

{\noindent\small
Martin Raum\\
Chalmers tekniska högskola och G\"oteborgs Universitet,
Institutionen för Matematiska vetenskaper,
SE-412 96 Göteborg, Sweden\\
E-mail: \url{martin@raum-brothers.eu}\\
Homepage: \url{https://martin.raum-brothers.eu}
}\vspace{.5\baselineskip}

{\noindent\small
Olav K. Richter\\
Department of Mathematics,
University of North Texas,
Denton, TX 76203, USA\\
E-mail: \url{richter@unt.edu}\\
Homepage: \url{http://www.math.unt.edu/~richter/}
}

%%%%%%%%%%%%%%%%%%%%%%%%%%%%%%%%%%%%%%%%%%%%%%%%%%
%%% TODOS

%\ifdraft{%
%\listoftodos%
%}

\end{document}

%% vim: spell spelllang=en_us